\documentclass{article}
\usepackage{fancyhdr}

\usepackage[british]{babel}
\usepackage[letterpaper,top=2cm,bottom=2cm,left=3cm,right=3cm,marginparwidth=1.75cm]{geometry}
\usepackage{amsmath}
\usepackage{amssymb}
\usepackage{amsthm}
\usepackage{graphicx}
\usepackage{xspace}
\usepackage[T1]{fontenc}
\usepackage[utf8]{inputenc}
\usepackage{float}
\usepackage{caption}
\usepackage{quiver}
\usepackage{yhmath}
\usepackage[nottoc,numbib]{tocbibind}
\usepackage{scalerel}
\usepackage{footmisc}
\usepackage[useregional]{datetime2}
\usepackage[colorlinks=true, allcolors=blue]{hyperref}
\usepackage[maxbibnames=99, style=alphabetic,backend=biber, url=false, doi=false, isbn=false, date=year]{biblatex}
\usepackage{filecontents}
\DTMlangsetup[en-GB]{showdayofmonth=false}

\newtheorem{proposition}{Proposition}[section]
\newtheorem{theorem}[proposition]{Theorem}
\newtheorem{lemma}[proposition]{Lemma}

\newtheorem{definition}[proposition]{Definition}
\newtheorem{remark}[proposition]{Remark}

\newtheorem{corollary}[proposition]{Corollary}

\makeatletter
\newcommand{\etale}{\'etal\@ifstar{\'e}{e\xspace}}
\makeatother
\newcommand\blfootnote[1]{%
  \begingroup
  \renewcommand\thefootnote{}\footnote{#1}%
  \addtocounter{footnote}{-1}%
  \endgroup
}
\newcommand{\Addresses}{{
  \bigskip
\noindent\textsc{Department of Mathematics, University of California, Berkeley,}\par\nopagebreak
\noindent Email: \texttt{tong.zhou@berkeley.edu}
  }}
\newcommand{\sslash}{\mathbin{/\mkern-6mu/}}

\newcommand{\CC}{\mathbf{C}}
\newcommand{\ZZ}{\mathbf{Z}}

\newcommand{\CF}{\mathcal{F}}
\newcommand{\CG}{\mathcal{G}}
\newcommand{\CM}{\mathcal{M}}

\newcommand{\Gm}{\mathbf{G}_m}
\newcommand{\X}{\times}
\newcommand{\OX}{\otimes}
\newcommand{\BX}{\boxtimes}

\newcommand{\QQbarl}{\overline{\mathbf{Q}}_\ell}

\newcommand{\CH}{\mathcal{H}}
\newcommand{\CL}{\mathcal{L}}

\newcommand{\AAA}{\mathbf{A}}

\newcommand{\CO}{\mathcal{O}}
\newcommand{\CN}{\mathcal{N}}

\newcommand{\fl}{\mathfrak{l}}
\newcommand{\fb}{\mathfrak{b}}
\newcommand{\fu}{\mathfrak{u}}
\newcommand{\g}{\mathfrak{g}}
\newcommand{\fp}{\mathfrak{p}}
\newcommand{\fh}{\mathfrak{h}}
\newcommand{\ResGP}{\mathrm{Res}^G_P}
\newcommand{\IndPG}{\mathrm{Ind}^G_P}

\title{Character Sheaves on Reductive Lie Algebras in Positive Characteristic}
\author{Tong Zhou}
\date{}
\begin{document}
\maketitle
\begin{abstract}\blfootnote{May 2024}
We prove a microlocal characterisation of character sheaves on a reductive Lie algebra over an algebraically closed field of sufficiently large positive characteristic: a perverse irreducible $G$-equivariant sheaf is a character sheaf if and only if it has nilpotent singular support and is quasi-admissible. We also present geometric proofs, in positive characteristic, of the equivalence between being admissible and being a character sheaf, and various characterisations of cuspidal sheaves, following the work of Mirković.
\end{abstract}


\section{Introduction}
In \cite{lusztig_fourier_1987}, Lusztig defined and studied character sheaves on reductive Lie algebras over an algebraic closure of a finite field, which are analogues of his character sheaves on reductive groups (\cites{lusztig_intersection_1984, lusztig_character_1985}). Later, in the context of D-modules on reductive Lie algebras over $\CC$, Mirković (\cite{mirkovic_character_2004}) obtained analogues results using more geometric methods. He also obtained the following microlocal characterisation of character sheaves: 

\begin{theorem}[{\cite[\nopp 6.3]{mirkovic_character_2004}}]\label{thm_mirk_intro}
    Let $G$ be a connected reductive group over $\CC$, $\g$ be its Lie algebra. Then, an irreducible $G$-equivariant\footnote{In this paper, $G$-equivariance is with respect to the adjoint action unless otherwise specified.} holonomic D-module $\CM$ (concentrated in degree 0) on $\g$ is a character sheaf if and only if $\CM$ has nilpotent singular support and is quasi-admissible. 
\end{theorem}

The notions of character sheaves and quasi-admissible sheaves will be recalled in §\ref{sec_def_of_sheaves}. The latter roughly means that the restriction of $\CM$ to each stratum $S_{(L,\CO)}$ of the Lusztig stratification $\g=\bigsqcup S_{(L,\CO)}$ (see §\ref{sec_lusztigstrat}) takes a form which is particularly simple “in the centre direction”. Here $L$ is a Levi subgroup and $\CO$ is a nilpotent orbit of $L$ in its Lie algebra $\fl$.\\

For character sheaves on reductive groups over an algebraically closed field $k$, a similar microlocal characterisation (no quasi-admissibility is needed in the group case) is known for $k=\CC$ and D-modules by \cite{mirkovic_characteristic_1988}, and known in the “only if” direction for arbitrary characteristic and $\ell$-adic sheaves by \cite{psaromiligkos_character_2023}.\\

The main new result of this paper is the analogue of Theorem \ref{thm_mirk_intro} in positive characteristic:

\begin{theorem}[Theorem \ref{thm_characterise_char}]\label{thm_main_intro}
    Let $G$ be a connected reductive group over an algebraically closed field of characteristic $p>0$, $\g$ be its Lie algebra. Assume $p$ is sufficiently large\footnote{See Conventions (the third paragraph) for an explicit bound.} with respect to the root system of $G$. Then, a perverse irreducible $G$-equivariant $\QQbarl$-sheaf $\CF$ on $\g$ is a character sheaf if and only if $\CF$ has nilpotent singular support and is quasi-admissible. 
\end{theorem}

We outline the proof of the “if” direction and point out the novelties compared to \cite{mirkovic_character_2004}. Our general strategy is as in \textit{loc. cit.}: use the Lusztig stratification to obtain a cuspidal sheaf whose parabolic induction contains $\CF$ as an irreducible constituent. The candidate cuspidal sheaf will be obtained from the restriction of $\CF$ to $Z_r(\fl)+\CO$, where $(L,\CO)$ is such that $S_{(L,\CO)}$ 
is the “biggest” stratum on which $\CF$ is non-zero, $Z_r(\fl)$ is the regular part of the centre of $\fl$. Denote by $\CG$ the middle extension to $\fl$ of $\CF|_{Z_r(\fl)+\CO}$. One then needs to analyse the singular support of $\CG$. It is not clear whether Mirković's argument for this works in positive characteristic. Instead, we prove a proper-transversality result, which implies that the singular support of the pullback equals the pullback of the singular support (\cite[\nopp 8.15]{saito_characteristic_2017}). The nilpotency of the singular support of $\CF$ then implies that of $\CG$. The next crucial input is the following microlocal characterisation of cuspidal sheaves: for $G$ semisimple, a perverse irreducible $G$-equivariant sheaf on $\g$ is cuspidal if and only if it is supported on the nilpotent cone and has nilpotent singular support. It follows from this characterisation that irreducible constituents of $\CG$ are cuspidal. Finally, certain geometric properties of the Lusztig strata implies that $\CF$ is an irreducible constituent of the parabolic induction of some irreducible constituent of $\CG$, which finishes the proof. The aforementioned microlocal characterisation of cuspidal sheaves was previously known only over $\CC$. In this paper, we apply the compatibility of characteristic cycles and the Fourier transform for monodromic $\ell$-adic sheaves proved in \cite{zhou_fourier_2024} to obtain the same characterisation in positive characteristic (Theorem \ref{thm_characterise_cusp}.4).\\

Apart from Theorem \ref{thm_main_intro}, we also present proofs, in positive characteristic, of all other main results in \cite{mirkovic_character_2004} concerning admissible, character, and orbital sheaves.

\begin{theorem}[Theorems \ref{thm_characterise_cusp}, \ref{thm_adm=char}]\label{thm_main'_intro}
    Let $G, \g, p$ be as above (with $p$ sufficiently large), $\CF$ be a perverse irreducible $G$-equivariant sheaf on $\g$. Then:\\
    1) $\CF$ is admissible if and only if $\CF$ is a character sheaf.\\
    2) Assume further $G$ is semisimple. Then $\CF$ is cuspidal if and only if $\CF$ and its Fourier transform are supported on the nilpotent cone, if and only if $\CF$ is a character sheaf and is orbital, if and only if $\CF$ is supported on the nilpotent cone and has nilpotent singular support.
\end{theorem}

These results (except the last “if and only if”) are not new: they are the main theorems of \cite{lusztig_fourier_1987} (when $k$ is an algebraic closure of a finite field). We take this opportunity to write down details of the more geometric proofs in \cite{mirkovic_character_2004}, carried out in positive characteristic. We point out two differences between the proofs we present and those in \textit{loc. cit.}: a) in proving “parabolic restriction coincides with usual restriction” (Lemma \ref{lem_res=Res}), we use the Contraction Principle\footnote{This terminology is from \cite{drinfeld_compact_2015}, see §5.3 and Appendix C in \textit{loc. cit.}} and proper-transversal pullback; b) in proving character sheaves are admissible (Theorem \ref{thm_adm=char}), we use an induction argument. For b), Mirković's original argument depends on a claim about the irreducibility of a certain local system \cite[\nopp 5.8.ii]{mirkovic_character_2004}, we have not been able to verify this claim (see Remark \ref{rmk_mirk_claim}).\\

§\ref{sec_lusztigstrat} and §\ref{sec_prelim} are devoted to review and preliminary studies. §\ref{sec_lusztigstrat} concerns the Lusztig stratification, which is a crucial tool in the sequel. §\ref{sec_prelim} concerns parabolic induction and restriction, and the classes of sheaves which are the main objects of interest. In §\ref{sec_characterise_cusp} and §\ref{sec_characterise_char}, we prove Theorem \ref{thm_main'_intro} and Theorem \ref{thm_main_intro}, respectively.

\subsection*{Conventions}
     We fix an algebraically closed field $k$ of characteristic $p>0$. Fix another prime $\ell\neq p$. A variety means a finite type reduced separated scheme over $k$. For a variety $X$, $D(X)$ denotes $D^b_c(X,\QQbarl)$ (\cite[\nopp 1.1]{weil_II}). For $\CF\in D(X)$, by an irreducible constituent of $\CF$ we mean an irreducible subquotient of some $^p\CH^i(\CF)$. All derived categories are in the triangulated sense. All sheaf-theoretic functors are derived. A “sheaf” means an object of $D(X)$. A “local system” means an object of $D(X)$ whose cohomology sheaves are lisse with finite type stalks.\\
     
     If $H$ is an algebraic group over $k$ acting on $X$, the triangulated (resp. abelian) category of $H$-equivariant sheaves (resp. perverse sheaves) are denoted by $D_H(X)$ ($:=D^b_c([X/H],\QQbarl)$, $[X/H]$ being the quotient stack) (resp. $Perv_H(X)$). If $H$ is connected, we canonically identify $Perv_H(X)$ as a full subcategory of $Perv(X)$. Induction (resp. forgetful) functors are denoted by $\Gamma$ (resp. $For$) (with appropriate sup and sub scripts) ($For$ is left adjoint to $\Gamma$). For $x\in X$ a closed point, $H_x$ or $Z_H(x)$ denotes its stabiliser in $H$, $H(x)$ denotes its $H$-orbit, $Z_\fh(x)$ denotes elements commuting with $x$ in the Lie algebra $\fh$ of $H$. The set of isomorphism classes of irreducible objects in $Perv_H(X)$ is denoted by $IrrPerv_H(X)$. We will abuse notations and use $\CF\in IrrPerv_H(X)$ to mean $\CF$ is a representative in $Perv_H(X)$ of an element of $IrrPerv_H(X)$. \\
    
     $G$ denotes a fixed connected reductive linear algebraic group, $\g$ denotes its Lie algebra. Throughout this paper, we assume that $p$ does not divide the following set of integers: integer 2, the coefficient of the highest root of $L$, $|P(L)/Q(L)|$, and $|(X(T)/Q(L)_{tor}|$, for $L$ ranging through all Levi subgroups of $G$. Here $P(L)$, $Q(L)$, $X(T)$ denote the weight lattice, root lattice, and character group of $L$, respectively. It is known that under this assumption, the following statements are true (and remain true when $G$ is replaced by any Levi subgroup): 1) there exists a $G$-invariant non-degenerate symmetric bilinear form on $\g$; 2) for each semisimple $x\in\g$, $Z_G(x)^\circ$ is a Levi subgroup; 3) for each $x\in\g$, $Lie(Z_G(x))=Z_\g(x)$; 4) for each Levi subgroup $L$, the regular elements $Z_r(\fl):=\{x\in Z(\fl)\,|\, Z_G(x)^\circ=L\}\subseteq Z(\fl)$ and $\fl_r:=\{x\in\fl\,|\,Z_G(x_s)^\circ\subseteq L\}\subseteq \fl$ are open dense (see §\ref{sec_lusztigstrat}). We refer to \cite[\nopp 2.5, 2.6]{letellier_fourier_2005} for detailed discussion.\\ 

     We fix a $G$-invariant non-degenerate symmetric bilinear form on $\g$, and implicitly identify $\g$ and its dual using this form. This form restricts to an $L$-invariant non-degenerate symmetric bilinear form on $\fl$ for any Levi $L$ (\cite[\nopp 2.5.14]{letellier_fourier_2005}). $\CN_G$ denotes the nilpotent cone in $\g$. For any $x\in\g$,\footnote{In this paper, when talking about points in a Lie algebra, we always mean closed points.} we write $x=x_s+x_n$ for its Jordan decomposition.\\ 

     We fix a non-trivial continuous character $\psi: \ZZ/p \rightarrow \QQbarl^\X$. Fourier transforms are denoted by $F$ and are with respect to this character. As we work over an algebraically closed field, we may ignore Tate twists.\\

     We refer to \cites{beilinson_constructible_2016, saito_characteristic_2017, umezaki_characteristic_2020, barrett_singular_2023} for the theory of singular support and characteristic cycle of $\ell$-adic sheaves.

\subsection*{Acknowledgement}
I am grateful to David Nadler for many valuable discussions and for his initial suggestion to study character sheaves on Lie algebras. I would also like to thank Kostas Psaromiligkos for discussion on his paper \cite{psaromiligkos_character_2023}, and thank Pramod Achar and Xinchun Ma very much for discussions concerning Theorem \ref{thm_characterise_cusp}.4 (see the footnote there).

\section{The Lusztig stratification}\label{sec_lusztigstrat}
The setup is as in the Conventions. In this section, we review the positive characteristic analogue of the Lusztig stratification of a reductive Lie algebra introduced over $\CC$ in \cite[\nopp §6]{lusztig_cuspidal_1995}.\\

Let $(L,\CO)$ be a pair where $L$ is a Levi subgroup of $G$ and $\CO$ is a nilpotent orbit of $L$ in its Lie algebra $\fl$. $Z_r(\fl):=\{z\in Z(\fl)\,|\, Z_G(z)^\circ=L\}$ denotes the regular part of the centre of $\fl$. $\fl_r:=\{y\in\fl\,|\, Z_G(y_s)^\circ\subseteq L\}$ denotes the regular part of $\fl$, it contains $Z_r(\fl)+\CN_L$. By our assumption on $p$ (see Conventions), $Z_r(\fl)\subseteq Z(\fl)$ and  $\fl_r\subseteq \fl$ are open dense\footnote{That $Z_r(\fl)\subseteq Z(\fl)$ is open dense follows from the proof of \cite[\nopp 2.6.13.i]{letellier_fourier_2005}. That $\fl_r\subseteq \fl$ is open dense can be seen as follows: consider the natural map $\chi: \fl\rightarrow \fl\sslash L:= \mathrm{Spec}(Sym(\fl^{\vee})^L)$. We refer to \cite[\nopp 7.12, 7.13]{jantzen_nilpotent_2004} for basic facts about this map (note, in \textit{loc. cit.}, “$\chi$” denotes the composition of our $\chi$ with an embedding $\fl\sslash L\hookrightarrow \AAA^n$). As we assume $p\neq 2$, we have $\fl\sslash L\cong\fh\sslash W_L:= \mathrm{Spec}(Sym(\fh^{\vee})^{W_L})$, where $\fh$ is the Lie algebra of a fixed maximal torus of $L$, and $W_L$ is the Weyl group of $L$. $p\neq 2$ also implies that $\fh_r:=\fl_r\cap\fh$ is open dense in $\fh$ (c.f. \cite[\nopp 13.3]{jantzen_nilpotent_2004}). Let $\tilde{\fh}_r$ be the image of $\fh_r$ in $\fh\sslash W_L$ (it is open dense). $\fl_r\subseteq \fl$ being open dense then follows from the fact that $\fl_r$ equals $\chi^{-1}(\tilde{\fh}_r)$.}.

\begin{definition}[Lusztig strata]
    For a pair $(L,\CO)$, $S_{(L,\CO)}:=$ $^G(Z_r(\fl)+\CO)$ (the $G$-saturation of $Z_r(\fl)+\CO$ in $\g$) is called the \underline{Lusztig stratum} associated to $(L,\CO)$.
\end{definition}

Clearly, $S_{(L,\CO)}=$ $^G(Z_r(\fl)+\CO)$ only depends on the $G$-conjugacy class of $(L,\CO)$. Also note that the number of $G$-conjugacy classes of such pairs is finite. We emphasise that, unless otherwise specified, by “$(L,\CO)$” we always mean the actual pair, not the conjugacy class. 

\begin{proposition}\label{prop_lusztigstrat}
    1) For two pairs $(L,\CO)$ and $(L',\CO')$, if $S_{(L,\CO)}\cap S_{(L',\CO')}\neq\varnothing$, then  $(L,\CO)$ and $(L',\CO')$ are $G$-conjugate.\\
    2) Every $x\in\g$ lies in some $S_{(L,\CO)}$.\\
    3) $S_{(L,\CO)}$ is a locally closed, smooth, and irreducible subvariety of $\g$ of dimension $\mathrm{dim}(G/L)+\mathrm{dim}(Z(\fl))+\mathrm{dim}(\CO)$.\\
    4) The closure $\overline{S_{(L,\CO)}}$ equals the disjoint union of $S_{(L',\CO')}$, where $(L',\CO')$ ranges exactly once through each $G$-conjugacy class satisfying the following property: there exists a representative $(L',\CO')$ in this class such that $L$ lies in $L'$ and $\CO'$ is in the closure of the induced nilpotent orbit of $\CO$ in $\fl'$.\footnote{Let $P$ be any parabolic \textit{in $L'$} with Levi decomposition $P=L\ltimes U$. Then $\CO+\fu\subseteq \fl'$ is irreducible and consists of nilpotent elements, so there is a unique nilpotent orbit $\CO''$ of $L'$ in $\fl'$ such that $\CO''\cap(\CO+\fu)$ is dense in $\CO+\fu$. It is known that $\CO''$ does not depend on the choice of $P$, and is called the induced nilpotent orbit of $\CO$ in $\fl'$. See \cite{lusztig_induced_1979} for details.}
\end{proposition}

\begin{corollary}
    $\g=\bigsqcup S_{(L,\CO)}$ is a stratification, called the \underline{Lusztig stratification}. Here $(L,\CO)$ ranges exactly once through each $G$-conjugacy class.
\end{corollary}

We need some preparations before discussing the proof of the proposition. For each Levi $L$, let $W_L:=N_G(L)/L$ be its Weyl group. For each pair $(L,\CO)$, let $W_{(L,\CO)}$ be the subgroup of $W_L$ consisting of elements that maps $\CO$ to itself. Define $\Tilde{S}_{(L,\CO)}=\{(g,x)\in (G/L)\X \g\,|\, g^{-1}x\in Z_r(\fl)+\CO\}$ (note $\Tilde{S}_{(L,\CO)}$ depends on the actual pair $(L,\CO)$, not just its conjugacy class). It is isomorphic to $G\X^L(Z_r(\fl)+\CO)$ via $\Tilde{S}_{(L,\CO)}\rightarrow G\X^L(Z_r(\fl)+\CO): (g,x)\mapsto(g,g^{-1}x)$. As $G\X^L(Z_r(\fl)+\CO)$ is smooth, $\Tilde{S}_{(L,\CO)}$ is smooth. We have a commutative diagram 

\[\begin{tikzcd}
	{\tilde{S}_{(L,\mathcal{O})}} & {G\times^L(Z_r(\mathfrak{l})+\mathcal{O})} \\
	{\mathfrak{g}} & {\mathfrak{g}}
	\arrow["{pr_2}", from=1-1, to=2-1]
	\arrow["\sim", from=1-1, to=1-2]
	\arrow["a", from=1-2, to=2-2]
	\arrow[Rightarrow, no head, from=2-1, to=2-2]
\end{tikzcd}\]
where $a: (g,x)\mapsto gx$. Note that the images of $pr_2$ and $a$ both equal to $S_{(L,\CO)}$.\\

Let $P$ be any parabolic in $G$ with Levi decomposition $P=L\ltimes U$. The following lemma implies that $\Tilde{S}_{(L,\CO)}$ is also isomorphic to $\Tilde{S}'_{(L,\CO)}:=\{(g,x)\in (G/P)\X \g\,|\, g^{-1}x\in Z_r(\fl)+\CO+\fu\}$, and similarly as above, to $G\X^P(Z_r(\fl)+\CO+\fu)$.

\begin{lemma}[{\cite[\nopp 2.6.6]{letellier_fourier_2005}}]\label{lem_U_action_on_l_r}
    Let $P$ be any parabolic in $G$ with Levi decomposition $P=L\ltimes U$. Then, for any $x\in\fl_r$, $U\rightarrow \g: u\mapsto ux$ has image $x+\fu$, and $U$ is mapped isomorphically to its image. 
\end{lemma}

\begin{proof}[Proof of Proposition \ref{prop_lusztigstrat}]
    1) Let $x_0\in S_{(L,\CO)}\cap S_{(L',\CO')}$, then $x_0$ is $G$-conjugate to some $x\in Z_r(L)+\CO$ and $x'\in Z_r(L')+\CO'$. As $L=Z_G(x_s)^\circ$, $L'=Z_G(x'_s)^\circ$, $L$ and $L'$ are conjugate. We may thus assume $L'=L$. Let $g\in G$ be such that $x'=$ $^gx$, then $g\in N_G(L)$ (becasue $^gL=$ $^gZ_G(x_s)^\circ=Z_G(^gx_s)^\circ=L$). So $x_n$ and $x'_n$ are conjugate by an element of $N_G(L)$, which implies that the pairs $(L,\CO)$, $(L,\CO')$ are conjugate.\\

    2) Let $L=Z_G(x_s)^\circ$. Note that this is a Levi subgroup (see Conventions). Clearly $x\in\fl, x_s\in Z_r(\fl), x_s\in \CN_L$. So $x\in Z_r(\fl)+\CO$ for some nilpotent orbit $\CO\subseteq \CN_L$.\\

    3) This is \cite[\nopp 5.1.28]{letellier_fourier_2005}. It is claimed in \textit{loc. cit.} that the map $pr_2: \Tilde{S}_{(L,\CO)}\rightarrow S_{(L,\CO)}$ is a Galois covering with Galois group $W_{(L,\CO)}$. It is shown in \textit{loc. cit.} that $W_{(L,\CO)}$ acts freely on $\Tilde{S}_{(L,\CO)}$ (via translation on the first factor) and induces a map $\Tilde{pr_2}: \Tilde{S}_{(L,\CO)}/W_{(L,\CO)}\rightarrow S_{(L,\CO)}$ which is proper and bijective on closed points. To obtain the claim, we need to verify that, for any $(g,x)\in \Tilde{S}_{(L,\CO)}$, the tangent map $D(pr_2): (T\Tilde{S}_{(L,\CO)})_{(g,x)}\rightarrow (T\g)_x$ is injective. As discussed above, we may instead work with $\Tilde{S}'_{(L,\CO)}:=\{(g,x)\in (G/P)\X \g\,|\, g^{-1}x\in Z_r(\fl)+\CO+\fu\}\rightarrow Z_r(\fl)+\CO+\fu$. Without loss of generality, we may assume $(g,x)=(1,x)$. Consider  
\[\begin{tikzcd}
	G & {\tilde{S}'_{(L,\mathcal{O})}} \\
	& {\mathfrak{g}}
	\arrow["{pr_2}", from=1-2, to=2-2]
	\arrow["h", from=1-1, to=1-2]
	\arrow["f"', from=1-1, to=2-2]
\end{tikzcd}\]
where $f: g\mapsto gx$, $h: g\mapsto (g,gx)$. One can compute $D(f): \g\rightarrow(T\g)_x$, $y\mapsto [y,x]$, and $D(h): \g\rightarrow(T\Tilde{S}'_{(L,\CO)})_{(1,x)}$, $y\mapsto (y,[y,x])$. Choose a basis for $(T\Tilde{S}'_{(L,\CO)})_{(1,x)}$ consisting of $\mathrm{dim}(Z(\fl))+\mathrm{dim}(\CO)+\mathrm{dim}(\fu)$ vectors in the $Z_r(\fl)+\CO+\fu$ directions (in the $\g$ factor), and $\mathrm{dim}(G/P)$ vectors which are images under $D(h)$ of a basis for $\fu'$, where $U'$ is the unipotent radical of the opposite parabolic to $P$. $D(pr_2)$ is clearly injective in the $Z_r(\fl)+\CO+\fu$ directions, the above computation shows $D(pr_2)$ is also injective in the remaining $\mathrm{dim}(G/P)$ directions.\\

4) This is \cite[\nopp 6.5]{lusztig_cuspidal_1995}, with the same proof.
\end{proof}

The following lemma will be repeatedly used in the sequel.

\begin{lemma}\label{lem_intersect_lusz_strat}
    Let $(L,\CO)$ be a pair. Then:\\
    1) $S_{(L,\CO)}\cap\fl_r=$ $^{W_L}(Z_r(\fl)+\CO)=\bigsqcup_{\CO'}(Z_r(\fl)+\CO')$, where $\CO'\in\CN_L$ ranges through the nilpotent orbits $W_L$-conjugate to $\CO$.\\
    2) $\overline{S_{(L,\CO)}}\cap \fl_r=\bigsqcup_{\CO'}(S_{(L,\CO')}\cap\fl_r)=\bigsqcup_{\CO'}$ $^{W_L}(Z_r(\fl)+\overline{\CO})$, where $\CO'\in\CN_L$ ranges through the nilpotent orbits in $\overline{\CO}$.
\end{lemma}

\begin{proof}
    1) Let $x\in S_{(L,\CO)}\cap\fl_r$. There exists some $g\in G$ such that $^gx\in Z_r(\fl)+\CO$. On the one hand, $Z_G(x_s)^\circ\subseteq L$ by the definition of $\fl_r$; on the other hand $Z_G(^gx_s)^\circ=L$, implying that $\mathrm{dim}(Z_G(x_s)^\circ)=\mathrm{dim}(L)$. So $Z_G(x_s)^\circ=L$. Consequently $x_s\in Z_r(\fl)$ and $L=Z_G(^gx_s)^\circ=$ $^gZ_G(x_s)^\circ=$ $^gL$. So $g\in N_G(L)$, the claim follows.\\

    2) Let $x\in\overline{S_{(L,\CO)}}\cap\fl_r$. $x$ lies in some $S_{(L',\CO')}$, where $L'=Z_G(x_s)^\circ$ and $\CO'$ is some nilpotent orbit of $L'$. Then, as above, $L'\subseteq L$. But Proposition \ref{prop_lusztigstrat}.4 implies $\mathrm{dim}(L)\leq \mathrm{dim}(L')$, so $L=L'$. The claim then follows from the description of $\CO'$ in Proposition \ref{prop_lusztigstrat}.4.
\end{proof}

\section{Definitions and preliminary results}\label{sec_prelim}
The setup is as in the Conventions. In this section, we review parabolic induction and restriction, and the definitions and basic properties of various classes of sheaves on $\g$ from \cites{lusztig_fourier_1987, mirkovic_character_2004}.

\subsection{Parabolic induction and restriction}
\begin{definition}[parabolic induction and restriction]\label{def_para_ind_res}
    Let $L\subseteq P$ be a Levi and parabolic subgroup of $G$ with Levi decomposition $P=L\ltimes U$. The \underline{parabolic restriction} (from $G$ to $L\subseteq P$) is the functor $\ResGP:=\pi_!i^* For^G_P: D_G(\g)\rightarrow D_P(\fl)$, where $i$ and $\pi$ are as in the following diagram, and $P=L\ltimes U$ acts on $\fl$ by the unique action where $L$ acts by adjunction and $U$ acts trivially. Its right adjoint $\IndPG:=\Gamma^G_Pi_*\pi^!$ is called the \underline{parabolic induction}. Here $For^G_P$ (resp. $\Gamma^G_P$) is the forgetful (resp. induction) functor of equivariant sheaves.
\[\begin{tikzcd}
	& {\mathfrak{p}} \\
	{\mathfrak{g}} && {\mathfrak{l}} \\
	& {\mathfrak{g}/\mathfrak{u}}
	\arrow["i"', hook', from=1-2, to=2-1]
	\arrow["\pi", two heads, from=1-2, to=2-3]
	\arrow["{\pi'}"', two heads, from=2-1, to=3-2]
	\arrow["{i'}", hook', from=2-3, to=3-2]
	\arrow["\lrcorner"{anchor=center, pos=0.125, rotate=-45}, draw=none, from=1-2, to=3-2]
\end{tikzcd}\]
\end{definition}

Note that $L$, as a subgroup of $G$, is part of the given data, but we omit it in the notations for convenience. The following properties of parabolic inductions and restrictions are well-known:

\begin{lemma}\label{lem_Res_Ind_basics}
    1) $\ResGP$ and $\IndPG$ are perverse t-exact. In particular, they restrict to exact functors between $Perv_G(\g)$ and $Perv_P(\fl)$. Note that $Perv_P(\fl)=Perv_L(\fl)$.\\
    2) $\IndPG$ preserves semisimplicity of perverse sheaves of geometric origin.\\
    3) Parabolic inductions and restrictions are transitive. More precisely: let $L\subseteq P$ be a Levi and parabolic subgroup of $G$ with Levi decomposition $P=L\ltimes U$. Fix a maximal torus $T\subseteq L$ and a Borel $B\subseteq P$. Let $L_1\subseteq P_1$ be a Levi and parabolic of $L$, $P_1$ containing $L_1$ as a Levi factor, such that $(B\cap L)\subseteq P_1$. Note $P_1U$ is a parabolic of $G$ containing $L_1$ as a Levi factor. Then, for any $\CF\in Perv_{L_1}(\fl_1)$,  $\IndPG\mathrm{Ind}^L_{P_1}\CF\cong \mathrm{Ind}^G_{P_1U}\CF$. Similarly for parabolic restriction.\\
    4) $\ResGP$ and $\IndPG$ commute with Fourier transforms. More precisely: denote the Fourier transform on $\g$ (resp. $\fl$) by $F_\g$ (resp. $F_\fl$), then $\ResGP\circ F_\g=F_\fl\circ\ResGP$, $\IndPG\circ F_\fl=F_\g\circ\IndPG$.
\end{lemma}

\begin{proof}[Sketch of proofs]
    1) See, for example, \cite[\nopp 5.6]{bezrukavnikov_parabolic_2021}.\\ 2) This is straightforward, using the Decomposition Theorem.\\
    3) This is the analogue of \cite[\nopp 4.2]{lusztig_character_1985} in the Lie algebra case. The same argument, \textit{mutatis mutandis}, applies.\\
    4) Note that, under the $G$-invariant non-degenerate symmetric bilinear form, taking duals “reflects” the diagram above along the $\g$, $\fl$ axis. The formula for $\ResGP$ then follows easily from the compatibility of the Fourier transform with linear maps. The formula for $\IndPG$ follows from adjunction.
\end{proof}

\begin{lemma}\label{lem_res=Res}
    Let $(L,\CO)$ be a pair, $P\subseteq G$ be any parabolic containing $L$ as a Levi factor. Let $\CF\in Perv_G(\g)$. Then $(\ResGP\CF)|_{\fl_r}\cong \CF|_{\fl_r}[2\mathrm{dim}(U)]$.
\end{lemma}

\begin{proof}\footnote{This argument is the Lie algebra analogue of \cite[Proof of Theorem 4.1]{ginzburg_induction_1993} in the group case.}
    By proper base change, $(\ResGP\CF)|_{\fl_r}=\pi_!(\CF|_{\fl_r+\fu})$. Fix a maximal torus $T\hookrightarrow L$. Choose a $\Gm\hookrightarrow T$ corresponding to a cocharacter whose pairings with simple roots in $L$ are 0 and pairings with roots in $U$ are $>0$. Consider the action of $\Gm$ on $Z_r(\fl)+\CO+\fu$ through $\Gm\hookrightarrow T$. $\CF|_{\fl_r+\fu}$ is equivariant with respect to this action. By the Contraction Principle (c.f. \cite[\nopp 5.3.2]{drinfeld_compact_2015} and the references therein), $\pi_!(\CF|_{\fl_r+\fu})\cong (\CF|_{\fl_r+\fu})|^!_{\fl_r}$, where $(-)|^!_{\fl_r}$ denotes the !-restriction to $\fl_r=\fl_r+\{0\}$. As $\CF|_{\fl_r+\fu}$ is $U$-equivariant, its singular support is contained in the union conormals of the $U$-orbits. By the fact that the $U$-orbits are exactly $x+\fu$ for $x\in\fl_r$ (see the discussion in the paragraph above the proof of Proposition \ref{prop_lusztigstrat}), $\fl_r\hookrightarrow \fl_r+\fu$ is $SS(\CF|_{\fl_r+\fu})$-transversal. So, by \cite[\nopp 1.6]{barrett_singular_2023}, $(\CF|_{\fl_r+\fu})|^!_{\fl_r}\cong ((\CF|_{\fl_r+\fu})|_{\fl_r})\OX((\underline{\QQbarl}_{\fl_r+\fu})|^!_{\fl_r})=\CF|_{\fl_r}[2\mathrm{dim}(U)]$.
\end{proof}

\subsection{Various classes of sheaves}\label{sec_def_of_sheaves}
\begin{definition}[cuspidal sheaves]\label{def_cusp}
    $\CF\in IrrPerv_G(\g)$ is called \underline{cuspidal} if it is non-zero, and satisfies the following two properties:\\ 
    1) $\ResGP\CF=0$ for all $L\subseteq P$ in $G$, $P$ containing $L$ as a Levi factor;\\
    2) $\CF$ is of the form $AS\BX\CF'$ on $\g=Z(\g)\X[\g,\g]$ ($AS$ and $\CF'$ depend on $\CF$). Here $AS$ (short for Artin-Schreier) is the Fourier transform of a skyscraper sheaf $\underline{\QQbarl}_x$ for some closed point $x$ of $Z(\g)^{\vee}$, and $\CF'$ is some object in $Perv_G([\g,\g])$. 
\end{definition}

\begin{remark}
    It is easy to see that $AS\BX\CF'$ is cuspidal (with respect to $G$) if and only if $\CF'$ is cuspidal (with respect to $[L,L]$).
\end{remark}

\begin{definition}[admissible sheaves]
    $\CF\in IrrPerv_G(\g)$ is called \underline{admissible} if it is non-zero, and is the irreducible constituent of $\IndPG\CG$, for some $L\subseteq P$ in $G$, $P$ containing $L$ as a Levi factor, and some cuspidal sheaf $\CG$ of $L$.
\end{definition}

\begin{definition}[orbital and character sheaves]
    Let $\CF\in IrrPerv_G(\g)$. It is called \underline{orbital} if it is nonzero and its support is the closure of a single $G$-orbit. It is called a \underline{character sheaf} if it is non-zero and its Fourier transform is orbital.
\end{definition}

The name “character sheaf” will be justified by Proposition \ref{prop_char=Forb}, which makes manifest the similarity to character sheaves on a reductive group (c.f. \cite[\nopp 2.1]{mirkovic_characteristic_1988}). Before stating the proposition, we need some preparations.

\begin{lemma}\label{lem_res_ind_char_orb}
    Fix $L\subseteq P\subseteq G$, $P$ containing $L$ as a Levi factor. Then, the irreducible constituents of the image under $\IndPG$ of a character (resp. orbital) sheaf are character (resp. orbital) sheaves. The same holds for $\ResGP$.
\end{lemma}

\begin{proof}
    As parabolic inductions and restrictions commute with Fourier transforms (Lemma \ref{lem_Res_Ind_basics}.4), it suffices to prove the statement for orbital sheaves.\\
    
    Let $\CF$ be orbital, we show the irreducible constituents of $\IndPG(\CF)$ are orbital. If $supp(\CF)=\overline{G(x)}$ for some $x\in\g$, then the support of $\ResGP\CF=\pi_!(\CF|_\fp)[2\mathrm{dim}(L)-2\mathrm{dim}(U)]$ (notations as in Definition \ref{def_para_ind_res}, $For^G_P$ omitted) is contained in $\pi(\overline{G(x)}\cap \fp)$. Note that the closure of a $G$-orbit is a union of finitely many $G$-orbits (c.f. \cite[proof of 8.4]{jantzen_nilpotent_2004}), and for each orbit $G(y)$, $\pi(G(y)\cap \fp)$ consists of finitely many $L$-orbits, by Lemma \ref{lem_fin_conj_classes}.1. So each irreducible constituent of $\ResGP(\CF)$ is orbital.\\

    Conversely, let $\CF$ on $\fl$ be an orbital sheaf of $L$ with $supp(\CF)=\overline{L(x)}$ for some $x\in L$. Then the support of $\IndPG\CF=\Gamma^G_Pi_*\pi^!(\CF)$ is contained in $\overline{^G(\pi^{-1}\overline{L(x)})}$. As $\overline{L(x)}$ is a union of finitely many $L$-orbits, by Lemma \ref{lem_fin_conj_classes}.2, there are only finitely many $G$-conjugacy classes of the semisimple parts of elements in $\pi^{-1}\overline{L(x)}$. It follows (by the finiteness of the number of nilpotent orbits) that there are only finitely many $G$-orbits in $^G(\pi^{-1}\overline{L(x)})$, hence in $\overline{^G(\pi^{-1}\overline{L(x)})}$. So each irreducible constituent of $\IndPG(\CF)$ is orbital.
\end{proof}

We used the following lemma in the proof of Lemma \ref{lem_res_ind_char_orb}.

\begin{lemma}\label{lem_fin_conj_classes}
    Fix $L\subseteq P\subseteq G$, with Levi decomposition $P=L\ltimes U$.\\ 
    1) Let $G(x)$ be a $G$-orbit in $\g$. Then $\pi(G(x)\cap \fp)$ consists of finitely many $L$-orbits, where $\pi: \fp=\fl\ltimes\fu\rightarrow\fl$ is the projection.\\
    2) For any $x\in \fl$ and $u\in \fu$, the semisimple part of $x+u$ is $U$-conjugate to the semisimple part of $x$.
\end{lemma}

\begin{proof}
    1) \cite[\nopp 3.3]{mirkovic_character_2004} proved this for the case $T\subseteq B\subseteq G$. The general case reduces to this case as follows: fix $T\subseteq B\subseteq G$. By conjugation, we may assume that our $L$ and $P$ are standard, and that we have the following diagram, where $\overline{\fb}$ is a Borel subalgebra of $\fl$:
\[\begin{tikzcd}
	& {\mathfrak{b}} & {\mathfrak{p}} \\
	{\mathfrak{h}} & {\overline{\mathfrak{b}}} & {\mathfrak{l}}
	\arrow["{\overline{\pi}}", two heads, from=1-2, to=2-2]
	\arrow[hook, from=1-2, to=1-3]
	\arrow["\pi", two heads, from=1-3, to=2-3]
	\arrow[hook, from=2-2, to=2-3]
	\arrow["\lrcorner"{anchor=center, pos=0.125}, draw=none, from=1-2, to=2-3]
	\arrow["{\overline{\alpha}}", two heads, from=2-2, to=2-1]
	\arrow["\alpha"', two heads, from=1-2, to=2-1]
\end{tikzcd}\]

Then, $(\pi(G(x)\cap\fp))\cap\overline{\fb}=\overline{\pi}(G(x)\cap\fb)$. By \cite[\nopp 3.3]{mirkovic_character_2004}, $\overline{\alpha}((\pi(G(x)\cap\fp))\cap\overline{\fb})=\alpha(G(x)\cap\fb)$ is a finite set. As $\pi(G(x)\cap\fp)$ is clearly $L$-stable, it is a union of finitely many $L$-orbits. This, plus $\overline{\alpha}((\pi(G(x)\cap\fp))\cap\overline{\fb})$ being finite, imply that the number of $L$-conjugacy classes of semisimple parts of elements of $\pi(G(x)\cap\fp)$ is finite, which in turn (by the finiteness of the number of nilpotent orbits) implies that there are only finitely many $L$-conjugacy classes in $\pi(G(x)\cap\fp)$.\\

    2) This is the analogue of \cite[\nopp 5.1]{lusztig_intersection_1984} in the Lie algebra case. The same argument, \textit{mutatis mutandis}, applies.
\end{proof}

In $G$, fix a maximal torus $T$ and a Borel $B$ with $B=T\ltimes U$, $U$ being the unipotent radical of $B$. Denote their Lie algebras by $\fh$, $\fb$, and $\fu$ respectively. The $G$-invariant non-degenerate symmetric bilinear form on $\g$ canonically identifies the dual of the diagram 
$\begin{tikzcd}
	{\mathfrak{h}} & {\mathfrak{b}} & {\mathfrak{g}}
	\arrow["\alpha"', two heads, from=1-2, to=1-1]
	\arrow["\beta", hook, from=1-2, to=1-3]
\end{tikzcd}$ with
$\begin{tikzcd}
	{\mathfrak{h}} & {\mathfrak{g}/\mathfrak{u}} & {\mathfrak{g}}
	\arrow["{\beta'}"', two heads, from=1-3, to=1-2]
	\arrow["{\alpha'}", hook, from=1-1, to=1-2]
\end{tikzcd}$. Consider the following subset $\mathcal{C}(T,B)$ of $IrrPerv_G(\g)$: $\CF\in IrrPerv_G(\g)$ is in $\mathcal{C}(T,B)$ if it is non-zero and is an irreducible constituent of $\Gamma^G_B F_{\fb}(\CG)$, for some $\CG\in Perv_B(\fb)$ supported on a set of the form $\alpha^{-1}(S)$, $S$ being some finite set of closed points in $\fh$. 

\begin{proposition}\label{prop_char=Forb}
     $\mathcal{C}(T,B)$ coincides with the set of character sheaves.
\end{proposition}

\begin{proof}
    This is \cite[\nopp 3.5]{mirkovic_character_2004}, the same proof applies. 
\end{proof}

Admissible, character, and orbital sheaves are our primary objects of interest. The rest of this section is devoted to showing a key geometric property of character sheaves (Corollary \ref{cor_char_const_lusz_strat}). It is useful to introduce the following auxiliary class of sheaves:

\begin{definition}[quasi-admissible sheaves]\label{def_qadm}
    $\CF\in IrrPerv_G(\g)$ is called \underline{quasi-admissible} if it is non-zero, and for each $(L,\CO)$, each irreducible constituent of $\CF|_{Z_r(\fl)+\CO}$ is of the form $AS|_{Z_r(\fl)}\BX\CL'$, where $AS$ is as in Definition \ref{def_cusp} (with $G$, $\g$ replaced by $L$, $\fl$), and $\CL'$ is a local system in $Perv_L(\CO)$.
\end{definition}

\begin{lemma}\label{lem_qadm_const_lusz_strat}
    Quasi-admissible sheaves are constructible with respect to the Lusztig stratification.
\end{lemma}

\begin{proof}
    Let $(L,\CO)$ be a pair, $S_{(L,\CO)}=$ $^G(Z_r(\fl)+\CO)$ be the associated Lusztig stratum. By definition, for a quasi-admissible sheaf $\CF$, each irreducible constituent of $\CF|_{Z_r(\fl)+\CO}$ is of the form $AS|_{Z_r(\fl)}\BX\CL'$ (notations as in Definition \ref{def_qadm}). In particular, $\CF|_{Z_r(\fl)+\CO}$ is a local system. Since $S_{(L,\CO)}$ is the saturation of $Z_r(\fl)+\CO$ and $\CF|_{S_{(L,\CO)}}$ is $G$-equivariant, $\CF|_{S_{(L,\CO)}}$ is also a local system by Lemma \ref{lem_loc_sys_sat_equi}.
\end{proof}

\begin{lemma}\label{lem_loc_sys_sat_equi}
    Let $H$ be an algebraic group acting on an irreducible variety $X$, $Z$ be a locally closed irreducible subvariety of $X$. If  $\CF\in D_H(X)$ is a local system on $Z$, then it is a local system on $X$.
\end{lemma}

\begin{proof}
    Let $X_1$ be the maximal open subset of $X$ on which $\CF$ is a local system. It is necessarily $H$-stable. Let $X_2$ be the maximal open subset of $X-X_1$ on which $\CF$ is a local system. It is also $H$-stable. Iterate this process, we get a $G$-stable stratification $X=X_1\sqcup X_2\sqcup...$ with respect to which $\CF$ is constructible, and each stratum intersects $Z$ (since $X$ is the saturation of $Z$). Assume $X_2$ is non-empty, we derive a contradiction as follows.\\

    Consider $X'=X_1\sqcup X_2$. As $\CF$ is not a local system on $X'$, there exist a specialisation of geometric points $x\leadsto y$ such that $\CF_x\leftarrow \CF_y$ is not an isomorphism (c.f. \cites[\nopp 0GJ2]{Stacks}[\nopp 4.4]{hansen_relative_2023}). By $G$-equivariance, we may assume $y$ is a geometric point of $X_2$. Then $x$ must be a geometric point of $X_1$ (as $\CF$ is locally constant on $X_2$). Consider the following specialisation diagrams: 
\[\begin{tikzcd}
	{\overline{\eta}_{X_1}} & x &&& {\mathcal{F}_{\overline{\eta}_{X_1}}} & {\mathcal{F}_x} \\
	{\overline{\eta}_Z} && y && {\mathcal{F}_{\overline{\eta}_{Z}}} && {\mathcal{F}_y}
	\arrow[squiggly, from=2-1, to=2-3]
	\arrow[squiggly, from=1-1, to=2-1]
	\arrow[squiggly, from=1-1, to=1-2]
	\arrow[squiggly, from=1-2, to=2-3]
	\arrow["a", from=2-5, to=1-5]
	\arrow["c"', from=1-6, to=1-5]
	\arrow["d"', from=2-7, to=1-6]
	\arrow["b", from=2-7, to=2-5]
\end{tikzcd}\]
where $\overline{\eta}_{X_1}$ (resp. $\overline{\eta}_Z$) is a geometric generic point of $X_1$ (resp. $Z$). We may assume the diagrams commute. It is clear that $a$, $b$, $c$ are isomorphisms. So $d$ is also an isomorphism, which is a contradiction.
\end{proof}

\begin{remark}
    The lemma is false if without the assumption that $Z$ is irreducible. For example, consider the additive group $H=\AAA^1$ acting on $X=\AAA^2=\mathrm{Spec}(k[x_1,x_2])$ via translation in the $x_2$-variable, and take $Z=(\{x_2=0\}-\{(0,0)\})\sqcup \{(0,1)\}$.
\end{remark}

\begin{lemma}\label{lem_char_are_qadm} 
    Character sheaves are quasi-admissible.
\end{lemma}

\begin{proof}
    Let $\CF$ be a character sheaf, $(L,\CO)$ be a pair, $P$ be any parabolic containing $L$ as a Levi factor. By Lemma \ref{lem_res=Res}, $\CF|_{Z_r(\fl)+\CO}[2\mathrm{dim}(U)]\cong (\ResGP\CF)|_{Z_r(\fl)+\CO}$. By Lemma \ref{lem_res_ind_char_orb}.2, each irreducible constituent of $\ResGP\CF$ is a character sheaf, hence of the form $AS\BX\CF'$, where $AS$ is the Fourier transform of a skyscraper sheaf $\underline{\QQbarl}_x$ for some closed point $x$ of $Z(\fl)^{\vee}$, and $\CF'$ is some character sheaf on $[\fl,\fl]$. So each irreducible constituent of $(\ResGP\CF)|_{Z_r(\fl)+\CO}$ is the irreducible constituent of some $(AS|_{Z_r(\fl)})\BX(\CF'|_\CO)$. We claim that it must be of the form $(AS|_{Z_r(\fl)})\BX\CL'$ for some local system $\CL'$ on $\CO$. Indeed: $\CF'|_\CO$ is a local system by $L$-equivariance, denote its irreducible constituents by $\{\CL'_i\}$, then, as $(AS|_{Z_r(\fl)})$ is irreducible of rank 1, each $(AS|_{Z_r(\fl)})\BX\CL'_i$ is irreducible, and $\{(AS|_{Z_r(\fl)})\BX\CL'_i\}$ are exactly the irreducible constituents of $(AS|_{Z_r(\fl)})\BX(\CF'|_\CO)$.
\end{proof}

Lemma \ref{lem_qadm_const_lusz_strat} and Lemma \ref{lem_char_are_qadm} imply:

\begin{corollary}\label{cor_char_const_lusz_strat}
    Character sheaves are constructible with respect to the Lusztig stratification.
\end{corollary}

\section{Characterisations of cuspidal sheaves, equivalence of being a character sheaf and being admissible}\label{sec_characterise_cusp}
The setup is as in the Conventions. In this section, we give several characterisations of cuspidal sheaves, and show the equivalence of being a character sheaf and being admissible.

\begin{theorem}\label{thm_characterise_cusp}
    Assume $G$ is semisimple. Then, for a non-zero $\CF\in IrrPerv_G(\g)$, the following are equivalent:\\
    1) $\CF$ is cuspidal;\\
    2) $\CF$ and $F\CF$ are supported on $\CN_G$;\\
    3) $\CF$ is a character sheaf and is orbital;\\
    4) $\CF$ is supported on $\CN_G$ and has nilpotent singular support (i.e. $SS(\CF)\subseteq \g\X\CN_G$).
\end{theorem}

\begin{proof}
    1) $\Rightarrow$ 2). By Lemma \ref{lem_res=Res} and the definition of cuspidal sheaves, $\CF|_{Z_r(\fl)+\CO}=0$ for all pairs $(L,\CO)$ except if $L=G$. For any $x\in\g$ with $\CF_x\neq 0$, we have $x\in Z_r(\fl)+\CO$, for $L=Z_\g(x_s)^\circ$ and $\CO$ some nilpotent orbit of $L$ (see the proof of Proposition \ref{prop_lusztigstrat}.2). So $Z_\g(x_s)^\circ=G$, $x\in Z_r(\g)+\CO$. By the semisimplicity of $G$, $Z(\g)=0$, so $x\in\CN_G$. So $supp(\CF)\subseteq\CN_G$.\\
    
    2) $\Rightarrow$ 1). Let $L\subseteq P\subsetneqq G$, $P$ containing $L$ as a Levi factor, we want to show $\ResGP\CF=0$. First note $supp(\ResGP\CF)\subseteq \CN_L$. Indeed, recall $\ResGP\CF=\pi_!(\CF|_\fp)[2\mathrm{dim}(U)]$ (notations as in Definition \ref{def_para_ind_res}), if $x\in \fl$ is such that $(\ResGP\CF)_x\neq 0$, then there exists $u\in\fu$ such that $\CF|_{x+u}\neq 0$. So $x+u\in \CN_G$. Since the semisimple part of $x+u$ (which is 0) is $U$-conjugate to the semisimple part of $x$ (Lemma \ref{lem_fin_conj_classes}.2), $x\in \CN_L$. This shows $supp(\ResGP\CF)\subseteq \CN_L$.\\

    For the same reason, $supp(\ResGP F\CF)\subseteq \CN_L$. Since $Z(\fl)$ is non-trivial (as $P\neq G$) and $\CN_L\subseteq [\fl,\fl]$, $supp(\ResGP\CF)$ and $supp(\ResGP F\CF)$ cannot both be in $\CN_L$ unless both $\ResGP\CF$ and $\ResGP F\CF$ are 0.\\
    
    2) $\Rightarrow$ 3). $\CF$ is orbital because it is irreducible and supported on finitely many $G$-orbits, so there exists a single orbit whose closure is $supp(\CF)$. Similarly, $F\CF$ is orbital. So $\CF$ is a character sheaf.\\
    
    3) $\Rightarrow$ 2). As the statement is symmetric with respect to $F\CF$, it suffices to show $supp(\CF)\subseteq \CN_G$. By Corollary \ref{cor_char_const_lusz_strat}, $\CF$ is of the form $j_{!*}\CL$, for some non-zero local system $\CL$ on some Lusztig stratum $S_{(L,\CO)}$. Note $Z_r(\fl)+\CO\subseteq S_{(L,\CO)}$ and all $G$-orbits in $S_{(L,\CO)}$ are of the same dimension, $\CF$ being orbital then implies that $Z(\fl)$ must be trivial, i.e, $L=G$. So $supp(\CF)=\overline{^G\CO}\subseteq \CN_G$. \\
    
    2) $\Leftrightarrow$ 4).\footnote{In a previous version of this article, for this proof, we used \cite[Lemma 8.2.7.3]{achar_perverse_2021} which says, in particular, that every $\CF\in IrrPerv_G(\g)$ supported on $\CN_G$ is $\Gm$-equivariant with respect to the (weight-$1$) scaling action. We later realised that this is false (consider $G=\mathbf{SL}_2$). Achar pointed out that Lemma 8.2.7.3 becomes correct if $\Gm$ acts via the weight-2 scaling instead ($\lambda(x)=\lambda^2x$), and pointed us to \cite[beginning of §2.2]{achar_calculations_2019}. We thank Achar and Ma very much for discussions around this.\\
    \indent\indent  Our proof of the \underline{Claim} is inspired by \textit{loc. cit.} An alternative approach, which avoids using \cite[\nopp 5.3]{jantzen_nilpotent_2004}, is to assume that the characteristic $p$ is larger than the orders of the groups of components of the stabilisers of points in nilpotent orbits. Then $\CF$ as in the \underline{Claim} will be tame, monodromicity follows.} \underline{Claim}: every $\CF\in IrrPerv_G(\g)$ supported on $\CN_G$ is \textit{monodromic}\footnote{Recall that a sheaf $\CF$ on a finite dimensional vector space $V$ is called monodromic if the restriction of all $\CH^i(\CF)$ to all $\Gm$-orbits (scaling action) are tame local systems. We refer to \cites{verdier_specialisation_1983, zhou_fourier_2024} for more details.}. Accepting this claim, applying the compatibility of characteristic cycles and Fourier transforms for monodromic sheaves (\cite[\nopp 1.2]{zhou_fourier_2024}), we get $SS(\CF)=SS(F\CF)$. The equivalence of 2) and 4) is then evident, using the fact that the base of the singular support equals the support.\\

    We now show the claim. Let $\CF\in IrrPerv_G(\g)$ be supported on $\CN_G$, and $\CO\subseteq \CN_G$ be a nilpotent orbit. By $G$-equivariance, $\CH^i(\CF)|_\CO$ is a local system concentrated in degree 0. We want to show $\CH^i(\CF)|_\CO$ is tame when restricted to the subvariety $k^\X x$, for all closed points $x\in\CO$. By \cite[\nopp 5.3]{jantzen_nilpotent_2004} (and our assumption on the characteristic $p$), there exists a cocharacter $\theta: \Gm\rightarrow G$ \textit{associated} to $x$ (terminology as in \textit{loc. cit.}). In particular, $\forall \lambda\in k^\X$, we have $(\ast): \theta(\lambda)(x)=\lambda^2x$. Now, $\CH^i(\CF)|_\CO$ is $G$-equivariant, hence $\Gm$-equivariant with $\Gm$ acting via $\theta$ and the $G$-action. As $k^\X x$ is an orbit of this $\Gm$-action, $\CH^i(\CF)|_{k^\X x}$ is $\Gm$-equivariant, with respect to the weight-2 $\Gm$-action (by $(\ast)$). In particular, $\CH^i(\CF)|_{k^\X x}$ is trivialised by a degree $2$ étale cover, hence tame.   
\end{proof}

\begin{corollary}
    Assume $G$ is semisimple. Then the Fourier transform of a cuspidal sheaf is cuspidal.
\end{corollary}

In the following, we return to the general case ($G$ reductive).

\begin{corollary}\label{cor_cusp_are_char}
     If $\CF\in IrrPerv_G(\g)$ is cuspidal, then it is a character sheaf, and has nilpotent singular support.
\end{corollary}

\begin{proof}
    By definition, $\CF$ is of the form $AS\BX\CF'$ (notations as in Definition \ref{def_cusp}), with $\CF'$ cuspidal with respect to $[G,G]$. Then $F\CF\cong F(AS)\BX F(\CF')\cong \underline{\QQbarl}_x\BX F(\CF')$, for some closed point $x\in Z(\g)$. Since $F(\CF')$ is orbital (by Theorem \ref{thm_characterise_cusp}.3), so is $\underline{\QQbarl}_x\BX F(\CF')$. So $\CF$ is a character sheaf. Since $SS(AS\BX\CF')=p^\circ SS(\CF')$, where $p: Z(\g)\X [\g,\g]\rightarrow [\g,\g]$ is the projection, and $SS(\CF')$ is nilpotent, $SS(AS\BX\CF')$ is also nilpotent.
\end{proof}

We now prove:

\begin{theorem}\label{thm_adm=char}
    $\CF\in IrrPerv_G(\g)$ is admissible if and only if it is a character sheaf.
\end{theorem}

\begin{proof}
    Let $\CF$ be an admissible sheaf, i.e., it is an irreducible constituent of $\IndPG\CG$ for some $L\subseteq P\subseteq G$, $P$ containing $L$ as a Levi factor, and some cuspidal $\CG$ on $\fl$. By Corollary \ref{cor_cusp_are_char} and Lemma \ref{lem_res_ind_char_orb}, $\CF$ is a character sheaf.\\

    Conversely, let $\CF$ be a character sheaf, we want to show it is admissible. If $G$ is a torus, this is clear. Assume the statement is true for all Levi factors of all proper parabolic subgroups of $G$, we show it is true for $G$. If $\CF$ is cuspidal, this is clear. If not, let $L\subseteq P\subsetneqq G$, with Levi decomposition $P=L\ltimes U$, be such that $\ResGP\CF\neq 0$. Let $\CG$ be an irreducible constituent of $\ResGP\CF$ which can be written as a quotient of $\ResGP\CF$. $\CG$ is a character sheaf by Lemma \ref{lem_res_ind_char_orb}. So, by the induction hypothesis, there exist $L_1\subseteq P_1$ of $L$, $P_1$ containing $L_1$ as a Levi factor, and $\CH\in IrrPerv_{L_1}(\fl_1)$ cuspidal, such that $\CG$ is an irreducible constituent of $\mathrm{Ind}^L_{P_1}\CH$. Note that $\CG$ is in fact a direct summand of $\mathrm{Ind}^L_{P_1}\CH$ by Lemma \ref{lem_Res_Ind_basics}.2.\footnote{Note that cuspidal sheaves are of geometric origin, because a cuspidal sheaf may be obtained as the middle extension of some irreducible constituent of the pushforward of the constant sheaf on some Galois cover of some nilpotent orbit (we are using the fact that, for the action of a connected algebraic group, an equivariant local system in degree 0 on an orbit corresponds to a representation of $\pi_0$ of the stabiliser).} To show $\CF$ is admissible, because of its irreducibility, it suffices to show $Hom(\CF, \mathrm{Ind}^G_{P_1}\CH)\neq 0$. By the transitivity of the parabolic induction (Lemma \ref{lem_Res_Ind_basics}.3 \footnote{The assumptions there are satisfied possibly after an $L$-conjugation, which is harmless for our purpose.}) and the adjunction, $Hom(\CF, \mathrm{Ind}^G_{P_1U}\CH)=Hom(\CF, \IndPG\mathrm{Ind}^L_{P_1}\CH)=Hom(\ResGP\CF, \mathrm{Ind}^L_{P_1}\CH)$. Since $Hom(\ResGP\CF, \CG)\neq 0$ (by construction, $\CG$ is a quotient of $\ResGP\CF$) and $Hom(\CG, \mathrm{Ind}^L_{P_1}\CH)\neq 0$ (because $\CG$ is a direct summand of $\mathrm{Ind}^L_{P_1}\CH$), we get $Hom(\CF, \mathrm{Ind}^G_{P_1U}\CH)\neq 0$.
\end{proof}

\begin{remark}\label{rmk_mirk_claim}
     Mirković's original proof of the “if” direction of Theorem \ref{thm_adm=char} depends on a claim about the irreducibility of a certain local system \cite[\nopp 5.8.ii]{mirkovic_character_2004}. Namely: for $\CF\in IrrPerv_G(\g)$ quasi-admissible, by Lemma \ref{lem_qadm_const_lusz_strat}, $\CF\cong j_{!*}\CL$ for some perverse irreducible local system on some $S_{(L,\CO)}$, he claimed $\CL|_{Z_r(\fl)+\CO}$ is irreducible. We have not been able to verify this claim, and use an induction argument instead. Note that, using the notations as in the discussion after Proposition \ref{prop_lusztigstrat}, $\CL|_{Z_r(\fl)+\CO}\cong (a^*\CL)|_{\{1\}\X(Z_r(\fl)+\CO)}$, and $a$ is a Galois cover. We do not know if $a^*\CL$ is irreducible.
\end{remark}

\section{The microlocal characterisation of character sheaves}\label{sec_characterise_char}
The setup is as in the Conventions. We emphasise that $\g$ and its dual are identified throughout, using the $G$-invariant non-degenerate symmetric bilinear form.\\

The goal of this section is to prove the following microlocal characterisation of character sheaves. The proof will be given after some preliminary studies.

\begin{theorem}\label{thm_characterise_char}
    (With assumptions as above) $\CF\in IrrPerv_G(\g)$ is a character sheaf if and only if $\CF$ has nilpotent singular support and is quasi-admissible.
\end{theorem}

\begin{definition}
    Denote by $\Lambda$ the following subset of $T^*\g$: $\{(x,y)\in\g\X\CN_G\,|\, y\in Z_\g(x)\}$.
\end{definition}

\begin{lemma}\label{lem_form_of_SS}
    1) $\Lambda$ is closed, Lagrangian, and is contained in $\cup_{all} (T^*_{S_{(L,\CO)}}\g)$, where the union ranges exactly once through each $G$-conjugacy class of $(L,\CO)$ pairs. It follows, for dimensional reasons, that $\Lambda$ is of the form $\cup_{some} (\overline{T^*_{S_{(L,\CO)}}\g})$, for some classes of $(L,\CO)$.\\
    2) If $\CF\in Perv_G(\g)$ has nilpotent singular support, then $SS(\CF)\subseteq \Lambda$. It follows, for dimensional reasons, that $SS(\CF)$ is of the form $\cup_{some} (\overline{T^*_{S_{(L,\CO)}}\g})$, for some classes of $(L,\CO)$.
\end{lemma}

\begin{proof}
    1) $\Lambda$ is closed because it is defined by a closed condition. To see it is Lagrangian, first note that, in general, for any $x\in\g$, $(T^*_{G(x)}\g)_x=Z_\g(x)$.\footnote{Proof: $(T^*_{G(x)}\g)_x=((TG(x))_x)^\perp=[\g,x]^\perp=Z_\g(x)$, where in the second equality we have used the fact that, under our assumption on $p$, $Lie(Z_G(x))=Z_\g(x)$ (see Conventions).} As $\Lambda$ coincides with $=\{(x,y)\in\g\X\CN_G\,|\, x\in Z_\g(y)\}\subseteq \g\X\g$. View $\g\X\g$ as the cotangent bundle of the second $\g$ factor, $\Lambda$ is then precisely $\cup_{\CO\subseteq\CN_G}\g$. This shows it is Lagrangian.\\
    
    We now show $\Lambda\subseteq\cup_{all} (T^*_{S_{(L,\CO)}}\g)$. For any $x\in\g$, let $L=Z_G(x_s)^\circ$. Then, by Lemma \ref{lem_conormal_comp}, $x\in Z_r(\fl)+\CO$ for some nilpotent orbit $\CO$ of $L$, and $(T^*_{S_{(L,\CO)}}\g)_x=Z_{[\fl,\fl]}(x_n)$. As $Z_\g(x)\cap\CN_G=Z_\fl(x)\cap\CN_G=Z_\fl(x)\cap\CN_L=Z_{[\fl,\fl]}(x_n)\cap\CN_L$ (we used $Z_\g(x)=Z_\fl(x)$, see the proof of Lemma \ref{lem_conormal_comp}), the claim follows.\\

    2) As $\CF$ is $G$-equivariant, $SS(\CF)\subseteq \cup_{x\in \g}((T^*_{G(x)}\g)_x)=\cup_{x\in \g}(Z_\g(x)_x)$. As $SS(\CF)\subseteq \g\X\CN_G$, $SS(\CF)\subseteq\{(x,y)\in\g\X\CN_G\,|\, y\in Z_\g(x)\}=\Lambda$.
\end{proof}

\begin{remark}
    We do not know if Lusztig stratification satisfies the analogue of Whitney condition A, i.e., if $\cup_{all} (T^*_{S_{(L,\CO)}}\g)$ is closed.
\end{remark}

\begin{lemma}\label{lem_conormal_comp}
    Let $x\in\g$, $L=Z_G(x_s)^\circ$. Then, $L$ is a Levi subgroup and $x\in Z_r(\fl)+\CO$ for some nilpotent orbit $\CO$ of $L$. Let $S_{(L,\CO)}$ be the corresponding Lusztig stratum. Then $(T^*_{Z_r(\fl)+\CO}\fl)_x=Z_{[\fl,\fl]}(x_n)$,  $(T^*_{S_{(L,\CO)}}\g)_x=Z_{[\fl,\fl]}(x_n)$.
\end{lemma}

\begin{proof}
    The first claim has already been proved in the proof of Proposition \ref{prop_lusztigstrat}.2. The first part of the second claim is clear: $(T^*_{Z_r(\fl)+\CO}\fl)_x=Z_\fl(x)\cap (Z(\fl)^{\perp \text{in $\fl$}})=Z_{[\fl,\fl]}(x_n)$. Finally, we compute $(T^*_{S_{(L,\CO)}}\g)_x$: first notice $(T^*_{G(x)}\g)_x=Z_\g(x)=Z_\fl(x)$, where the last equality follows from $Z_\g(x)\subseteq Z_\g(x_s)=\fl$. This, plus dimensional reasons (Proposition \ref{prop_lusztigstrat}.3), imply $(TS_{(L,\CO)})_x=(TG(x))_x\oplus(TZ(\fl))_x$. So $(T^*_{S_{(L,\CO)}}\g)_x=(T^*_{G(x)}\g)_x\cap (Z(\fl)^{\perp \text{in $\g$}})=Z_{[\fl,\fl]}(x_n)$.
\end{proof}

We now prove Theorem \ref{thm_characterise_char}.

\begin{proof}[Proof of Theorem \ref{thm_characterise_char}]
Let $\CF$ be a character sheaf. By Lemma \ref{lem_char_are_qadm}, it is quasi-admissible. By Theorem \ref{thm_adm=char}, it is the irreducible constituent of $\IndPG\CG=\Gamma^G_P i_*\pi^!(\CG)$ (notations as in Definition \ref{def_para_ind_res}), for some $L\subseteq P\subseteq G$, $P$ containing $L$ as a Levi factor, and some cuspidal $\CG$ on $\fl$. To show $\CF$ has nilpotent singular support, it suffices to show $\IndPG\CG$ does. $\CG$ has nilpotent singular support by Corollary \ref{cor_cusp_are_char}. Since $\pi$ is smooth and $i$ is a closed immersion, $SS(i_*\pi^!(\CG))=i_\circ\pi^\circ SS(\CG)$. It follows that $i_*\pi^!(\CG)$ also has nilpotent singular support (this is the consequence of two simple facts: $T^*_\fp\g=\fu$ and $\CN_L+\fu\subseteq \CN_G$). By the general formula $SS(\Gamma^G_P\CH)\subseteq \overline{G.SS(\CH)}$ for any $\CH\in D_G(\g)$, and the fact that $\g\X\CN_G$ is stable under the $G$-action, we see that $\IndPG\CG=\Gamma^G_P i_*\pi^!(\CG)$ has nilpotent singular support.\\

We now show the converse: given $\CF\in IrrPerv_G(\g)$ quasi-admissible with nilpotent singular support, we will find a cuspidal sheaf for some Levi $L$ whose parabolic induction (for any $P$ containing $L$ as a Levi factor) contains $\CF$ as an irreducible constituent. Apply Theorem \ref{thm_adm=char}, we get $\CF$ is a character sheaf.\\

By Corollary \ref{cor_char_const_lusz_strat}, there exists a unique stratum $S_{(L,\CO)}$ such that $\CL:=\CF|_{S_{(L,\CO)}}$ is an irreducible perverse $G$-equivariant local system, and $\CF\cong j_{!*}\CL$, where $j$ is the inclusion $S_{(L,\CO)} \hookrightarrow \g$. Let $\CL_0=\CL|_{Z_r(\fl)+\CO}[-2\mathrm{dim}(U)]$ (the shift ensures $\CL_0$ is perverse). Consider $j'_{!*}\CL_0\in Perv_L(\fl)$, where $j'$ is the inclusion $Z_r(\fl)+\CO \hookrightarrow \fl$. Fix any $P$ containing $L$ as a Levi factor.\\

\underline{Claim 1}: $\CF$ is an irreducible constituent of $\IndPG(j'_{!*}\CL_0)$.\\

\underline{Claim 2}: $SS(j'_{!*}\CL_0)\subseteq \fl\X\CN_L$.\\

Accepting these two claims for now, we produce a cuspidal sheaf on $\fl$ whose $\IndPG$ contains $\CF$ as an irreducible constituent. By Claim 1, $\CF$ is an irreducible constituent of the $\IndPG$ of some irreducible constituent of $j'_{!*}\CL_0$, which is necessarily of the form $j'_{!*}\CL_i$ for some irreducible constituent $\CL_i$ of $\CL_0$. By the quasi-admissibility of $\CF$, $\CL_i$ is of the form $AS|_{Z_r(\fl)}\BX\CL_i'$, for some $AS$ as in Definition \ref{def_cusp} (with $G$, $\g$ replaced by $L$, $\fl$), and some local system $\CL'_i$ on $\CO$, which is necessarily perverse, irreducible, and $L$-equivariant. We compute: $j'_{!*}\CL_i=j'_{!*}(AS|_{Z_r(\fl)}\BX\CL_i')\cong (j'_{!*}AS|_{Z_r(\fl)})\BX(j'_{!*}\CL_i')\cong AS\BX(j'_{!*}\CL_i')$ (we abuse notations and denote $Z_r(\fl)\hookrightarrow Z(\fl)$ and $\CO\hookrightarrow [\fl,\fl]$ also by $j'$). Claim 2 implies $SS(j'_{!*}\CL_i)\subseteq \fl\X\CN_L$. Combine this with the previous computation, we get $SS(j'_{!*}\CL'_i)\subseteq [\fl,\fl]\X\CN_L$. By Theorem \ref{thm_characterise_cusp}.4, $j'_{!*}\CL'_i$, hence $j'_{!*}\CL_i$, is cuspidal.\\

It remains to prove the two claims. For Claim 1: as $supp(\IndPG(j'_{!*}\CL_0))=\overline{S_{(L,\CO)}}$, it suffices to show $(\IndPG(j'_{!*}\CL_0))|_{S_{(L,\CO)}}$ is a local system containing $\CL$ as an irreducible constituent (using Lemma \ref{lem_perv_1}). Consider the induction diagram
\[\begin{tikzcd}
	{G\times\mathfrak{g}} & {G\times^P\mathfrak{g}} \\
	{\mathfrak{g}} & {\mathfrak{g}}
	\arrow["{pr''_2}"', from=1-1, to=2-1]
	\arrow["\nu''", from=1-1, to=1-2]
	\arrow["a''", from=1-2, to=2-2]
\end{tikzcd}\]

$\IndPG(j'_{!*}\CL_0)$ is the $a_*$ of the descent to $Perv_G(G\X^P\g)$ of $pr_2^*(i_*\pi^!(j'_{!*}\CL_0))$. Restrict to $S_{(L,\CO)}$, using proper base change for $a_*$, we reduce to the following diagram 
\[\begin{tikzcd}
	{G\times S_{(L,\mathcal{O})}} & {G\times^PS_{(L,\mathcal{O})}} \\
	{S_{(L,\mathcal{O})}} & {S_{(L,\mathcal{O})}}
	\arrow["{pr'_2}"', from=1-1, to=2-1]
	\arrow["\nu'", from=1-1, to=1-2]
	\arrow["a'", from=1-2, to=2-2]
\end{tikzcd}\]

$(\IndPG(j'_{!*}\CL_0))|_{S_{(L,\CO)}}$ is then the $a'_*$ of the descent to $Perv_G(G\X^P S_{(L,\CO)})$ of $pr_2^{'*}((i_*\pi^!(j'_{!*}\CL_0))|_{S_{(L,\CO)}})=pr_2^{'*}(i_*\pi^!\CL_0)$, where the last equality follows from the fact $S_{(L,\CO)}\cap (Z(\fl)+\overline{\CO}+\fu)=Z_r(\fl)+\CO+\fu$.\footnote{Let $y\in S_{(L,\CO)}\cap (Z(\fl)+\overline{\CO}+\fu)$, let $y=y_z+y_o+y_u$, $y_z\in Z(\fl)$, $y_o\in\overline{\CO}$, $y_u\in \fu$. The semisimple part of $y$ is $U$-conjugate to the semisimple part of $y_z+y_o$ (Lemma \ref{lem_fin_conj_classes}.2), which is just $y_z$. As $y\in S_{(L,\CO)}$, $y_z$ must be in $Z_r(\fl)$, so $y\in S_{(L,\CO)}\cap(Z_r(\fl)+\overline{\CO}+\fu)$. By Lemma \ref{lem_U_action_on_l_r}, $y$ is $U$-conjugate to $y_z+y_o$, and $S_{(L,\CO)}\cap(Z_r(\fl)+\overline{\CO}+\fu)$ is $U$-stable. Consequently $y_z+y_o\in S_{(L,\CO)}\cap (Z_r(\fl)+\overline{\CO})$. As $S_{(L,\CO)}\cap (Z_r(\fl)+\overline{\CO})=Z_r(\fl)+\CO$ by Lemma \ref{lem_intersect_lusz_strat}.2, we are done.} By the $P$-equivariance of $\CL$ and the fact that $^P(Z_r(\fl)+\CO)=Z_r(\fl)+\CO+\fu$ (which follows, again, from Lemma \ref{lem_U_action_on_l_r}), we get $\pi^!\CL_0\cong\CL|_{Z_r(\fl)+\CO+\fu}$. So we may further restrict to the following diagram
\[\begin{tikzcd}
	{G\times (Z_r(\mathfrak{l})+\mathcal{O}+\mathfrak{u})} & {G\times^P(Z_r(\mathfrak{l})+\mathcal{O}+\mathfrak{u})} \\
	{Z_r(\mathfrak{l})+\mathcal{O}+\mathfrak{u}} & {S_{(L,\mathcal{O})}} & {}
	\arrow["{pr_2}"', from=1-1, to=2-1]
	\arrow["\nu", from=1-1, to=1-2]
	\arrow["a", from=1-2, to=2-2]
\end{tikzcd}\]

Notice that $a''\nu''$ is the action map $G\X\g\rightarrow \g$, using the $G$-equivariance of $\CL$, we get $\nu^*a^*\CL\cong pr_2^*(\CL|_{Z_r(\fl)+\CO+\fu})$, i.e., $a^*\CL$ is the descent to $Perv_G(G\X^P (Z_r(\fl)+\CO+\fu))$ of $pr_2^*(\CL|_{Z_r(\fl)+\CO+\fu})$. So $(\IndPG(j'_{!*}\CL_0))|_{S_{(L,\CO)}}\cong a_*a^*\CL\cong \CL\OX a_*\underline{\QQbarl}$, where in the last step we have used the projection formula. Recall the discussion after Proposition \ref{prop_lusztigstrat}, we know $G\X^P(Z_r(\fl)+\CO+\fu)\cong G\X^L(Z_r(\fl)+\CO)$, and that $a$ is a Galois cover. So $a_*\underline{\QQbarl}$ is a local system, with $\underline{\QQbarl}$ as a direct summand. This proves Claim 1.\\

To see Claim 2, we need two further claims:\\

\underline{Claim 2.1}: $i: \fl_r\hookrightarrow \g$ is properly $SS(\CF)$-transversal, and $i^\circ SS(\CF)\subseteq \fl_r\X\CN_L$. The claim implies $SS(\CF|_{\fl_r})=i^\circ SS(\CF)$ is contained in $\fl_r\X\CN_L$.\footnote{We are using the compatibility of the singular support with properly transverse pullbacks, see \cites[\nopp 8.15]{saito_characteristic_2017}[\nopp 1.6]{barrett_singular_2023}.} Proof of the claim: first recall Lemma \ref{lem_form_of_SS}: $SS(\CF)\subseteq \Lambda \subseteq \cup_{all} (T^*_{S_{(L',\CO')}}\g)$ and $SS(\CF)=\cup_{some} \overline{(T^*_{S_{(L',\CO')}}\g)}$. By Lemma \ref{lem_intersect_lusz_strat}.2, it suffices to consider the pairs $(L',\CO')$ with $L=L'$, $\CO'\subseteq \overline{\CO}$. For such a pair, consider any $x\in S_{(L,\CO')}\cap \fl_r$. By conjugating $(L,\CO')$ by $W_L$, we move $x$ into $Z_r(\fl)+\CO'$ (Lemma \ref{lem_intersect_lusz_strat}). By Lemma \ref{lem_conormal_comp}, $(T^*_{S_{(L,\CO')}}\g)_x=(T^*_{Z_r(\fl)+\CO'}\fl)_x$. Consequently: 1) $S_{(L,\CO')}$ intersects $\fl_r$ transversely; 2) $i^\circ$, when restricted to $(T^*_{S_{(L,\CO')}}\g)_x$, is the “identity” map. In particular, it preserves nilpotent elements. These imply that $i$ is $SS(\CF)$-transversal, and that $\overline{(T^*_{S_{(L,\CO')}}\g)}\X_\g \fl_r$ has the correct dimension, so $i$ is properly $SS(\CF)$-transversal.\\

\underline{Claim 2.2}: $Z_r(\fl)+\CO$ is an open subset of $supp(\CF|_{\fl_r})$. Proof: it is clear that $S_{(L,\CO)}\cap \fl_r$ is open in $\overline{S_{(L,\CO)}}\cap \fl_r$. By Lemma \ref{lem_intersect_lusz_strat}.1, $S_{(L,\CO)}\cap \fl_r=$ $^{W_L}(Z_r(\fl)+\CO)=\bigsqcup_{\sigma\in W_L}$ $^\sigma(Z_r(\fl)+\CO)$.  So $Z_r(\fl)+\CO$ is open in $\overline{S_{(L,\CO)}}\cap \fl_r$. From $supp(\CF)=\overline{S_{(L,\CO)}}$, we know $supp(\CF|_{\fl_r})\subseteq \overline{S_{(L,\CO)}}\cap \fl_r$. But we also know $supp(\CF|_{\fl_r})$ contains $Z_r(\fl)+\CO$, so $Z_r(\fl)+\CO$ is an open subset of $supp(\CF|_{\fl_r})$.\\

We now prove Claim 2. We saw in the paragraph after Claim 2 that quasi-admissibility implies irreducible constituents of $j'_{!*}\CL_0$ are of the form $AS\BX(j'_{!*}\CL_i')$. So $SS(j'_{!*}\CL_0)$ is of the form $p^\circ C$ for some $C\subseteq T^*[\fl,\fl]$, where $p: \fl=Z(\fl)\X[\fl,\fl]\rightarrow [\fl,\fl]$ is the projection. This, plus the fact that $j'_{!*}\CL_0$ is supported on $Z(\fl)+\overline{\CO}$, implies that it suffices to prove Claim 2 after restricting to $\fl_r$. By Claim 2.2, every irreducible constituent of $(j'_{!*}\CL_0)|_{\fl_r}$ is an irreducible constituent of $\CF|_{\fl_r}$ (using Lemma \ref{lem_perv_1}), so $SS((j'_{!*}\CL_0)|_{\fl_r})\subseteq SS(\CF|_{\fl_r})$. By Claim 2.1, $SS(\CF|_{\fl_r})$, hence $SS((j'_{!*}\CL_0)|_{\fl_r})$, is contained in $\fl_r\X\CN_L$.
\end{proof}


We used the following lemma in the proof above.

\begin{lemma}\label{lem_perv_1}
    Let $X$ be a variety (not necessarily irreducible), $\CF\in D(X)$ with $supp(\CF)=X$. If $\CF$ is a local system when restricted to some open, smooth and irreducible $j: U\hookrightarrow X$, then for each irreducible constituent $\CL$ of $\CF|_U$, $j_{!*}\CL$ is an irreducible constituent of $\CF$.
\end{lemma}

\begin{proof}
    Using the perverse t-exactness of $j^*$, we easily reduce to the case where $\CF$ is perverse. Assume $\CF$ is perverse in the following. Let $X_1=\overline{U}$, it is the irreducible component that $U$ lies in. Shrinking $U$, we may assume $U$ does not intersect other irreducible components. $\CF$ is a successive extension of its irreducible constituents, each of which is of the form $j'_{!*}\CL'$, for some perverse irreducible local system $\CL'$ on some locally closed, smooth and irreducible $j': U'\hookrightarrow X$. Accordingly, $\CF|_U$ is a successive extension of $(j'_{!*}\CL')|_U$. We claim that $U'\cap U$, if non-empty, is open dense in $U$, in which case $(j'_{!*}\CL')|_U$ is an irreducible constituent $\CL$ of $\CF|_U$, and $j'_{!*}\CL'\cong j_{!*}\CL$. The lemma then follows easily.\\

    To see the claim, assume $U'\cap U$ is non-empty, then $U'$ must lie in $X_1$, then $U'\cap U$ is open dense in $U'$, so $(j'_{!*}\CL')|_U$ is still irreducible, and thus an irreducible constituent of $\CF|_U$. Since $\CF$ is a local system on $U$, $SS(\CF|_{U})$, hence $SS((j'_{!*}\CL')|_U)$, is contained in the zero section. This forces $U'\cap U$ to be open dense in $U$ and $(j'_{!*}\CL')|_U$ to be a local system. It is necessarily isomorphic to an irreducible constituent $\CL$ of $\CF|_U$, and we have $j'_{!*}\CL'\cong j_{!*}\CL$.
\end{proof}

\printbibliography

\Addresses

\end{document}